\documentclass[11pt,reqno]{amsart}
\usepackage[utf8]{inputenc}
\usepackage{amssymb,amsmath}
\usepackage{mathrsfs}
\usepackage[T1]{fontenc}
\usepackage{lmodern} \normalfont %to load T1lmr.fd
\DeclareFontShape{T1}{lmr}{bx}{sc} { <-> ssub * cmr/bx/sc }{}
\usepackage{microtype}
\usepackage{graphicx}
\usepackage{color}
\usepackage[dvipsnames]{xcolor}
\usepackage{tikz}
\usetikzlibrary{shapes,positioning}
\usepackage{pgfplots}
\pgfplotsset{compat=newest}
\usepackage{enumitem}
\usepackage{setspace}
\usepackage{mathtools}
\usepackage{xfrac}

\usepackage{multicol}
\usepackage{soul}
\usepackage{subcaption}
\usepackage{booktabs}
\usepackage{array}
\newcolumntype{L}[1]{>{\raggedright\let\newline\\\arraybackslash\hspace{0pt}}m{#1}}
\newcolumntype{C}[1]{>{\centering\let\newline\\\arraybackslash\hspace{0pt}}m{#1}}
\newcolumntype{R}[1]{>{\raggedleft\let\newline\\\arraybackslash\hspace{0pt}}m{#1}}

\usepackage[hidelinks]{hyperref}
\usepackage[nameinlink,noabbrev]{cleveref}
\newtheorem{theorem}{Theorem}
\newtheorem{lemma}[theorem]{Lemma}

\theoremstyle{definition}
\newtheorem{definition}[theorem]{Definition}
\newtheorem{example}[theorem]{Example}
\theoremstyle{remark}
\newtheorem{remark}[theorem]{Remark}

\newtheorem{assumption}[theorem]{Assumption}
\Crefname{assumption}{Assumption}{Assumptions}
\numberwithin{theorem}{section}
\numberwithin{table}{section}
\numberwithin{figure}{section}
\definecolor{myBlue}{RGB}{30,144,255} % dodger blue
\definecolor{myGreen}{RGB}{69,169,0} % chatreuse
\definecolor{myRed}{RGB}{165,12,42}
\definecolor{myOrange}{RGB}{225,92,22}
\definecolor{color0}{rgb}{0.12156862745098,0.466666666666667,0.705882352941177}
\definecolor{color1}{rgb}{1,0.498039215686275,0.0549019607843137}
\definecolor{color2}{rgb}{0.172549019607843,0.627450980392157,0.172549019607843}
\definecolor{color3}{rgb}{0.83921568627451,0.152941176470588,0.156862745098039}
\definecolor{color4}{rgb}{0.580392156862745,0.403921568627451,0.741176470588235}
\definecolor{color5}{rgb}{0,0,0}

\newcommand{\delete}[1]{ }

\def\R{\mathbb{R}}

\def\TT{\mathbb{T}}
\newcommand\dx{\,\text{d}x}

\newcommand{\ddt}{\ensuremath{\frac{\mathrm{d}}{\mathrm{d}t}} }

\newcommand{\calA}{\ensuremath{\mathcal{A}} }

\newcommand{\calB}{\ensuremath{\mathcal{B}} }

\newcommand{\calC}{\ensuremath{\mathcal{C}} }
\newcommand{\calD}{\ensuremath{\mathcal{D}} }

\newcommand{\cHV}{\ensuremath{\mathcal{H}_{\scalebox{.5}{\V}}}}
\newcommand{\cHQ}{\ensuremath{\mathcal{H}_{\scalebox{.5}{\Q}}}}

\newcommand{\calI}{\ensuremath{\mathcal{I}} }
\newcommand{\calK}{\ensuremath{\mathcal{K}} }
\newcommand{\calM}{\ensuremath{\mathcal{M}} }

\newcommand{\Q}{\ensuremath{\mathcal{Q}} }

\newcommand{\calR}{\ensuremath{\mathcal{R}} }

\newcommand{\calY}{\ensuremath{\mathcal{Y}} }
\newcommand{\V}{\ensuremath{\mathcal{V}}}

\newcommand{\calDD}{\ensuremath{\overline\calD}}
\newcommand{\calKK}[1]{\ensuremath{\overline\calK_{#1}}}
\newcommand{\calMM}{\ensuremath{\overline\calM}}
\newcommand{\state}{z}
\definecolor{corrColor}{RGB}{60,124,155}

\newcommand{\inpVar}{v}
\newcommand{\outVar}{y}

\newcommand{\Id}{\mathrm{Id}}
\allowdisplaybreaks
\begin{document}
\title{Port-Hamiltonian formulations of poroelastic network models}
\author[]{R.~Altmann$^\dagger$, V.~Mehrmann$^{\ddagger}$, B.~Unger$^{\star}$}
\address{${}^{\dagger}$ Department of Mathematics, University of Augsburg, Universit\"atsstr.~14, 86159 Augsburg, Germany}
\email{robert.altmann@math.uni-augsburg.de}
\address{${}^{\ddagger}$ Institute of Mathematics MA\,{}4-5, Technical University Berlin, Stra\ss e des 17.~Juni 136, 10623 Berlin, Germany}
\email{mehrmann@math.tu-berlin.de}
\address{${}^{\star}$ Stuttgart Center for Simulation Science (SC SimTech), University of Stuttgart, Universitätsstr.~32, 70569 Stuttgart, Germany}
\email{benjamin.unger@simtech.uni-stuttgart.de}
%\thanks{}
%
\date{\today}
%\keywords{poroelastic network model, port-Hamiltonian system, differential-algebraic operator equation, dissipative Hamiltonian system}
%
%
%=============================================================================
%=========  Abstract
%=============================================================================
\begin{abstract}
We investigate an energy-based formulation of the two-field poroelasticity model and the related multiple-network model as they appear in geosciences or medical applications.
We propose a port-Hamiltonian formulation of the system equations, which is beneficial for preserving important system properties after discretization or model-order reduction. For this, we include the commonly omitted second-order term and consider the corresponding first-order formulation. The port-Hamiltonian formulation of the quasi-static case is then obtained by (formally) setting the second-order term zero. 
Further, we interpret the poroelastic equations as an interconnection of a network of submodels with internal energies, adding a control-theoretic understanding of the poroelastic equations.
\end{abstract}
%
%
%=============================================================================
%=========  Title / Contents
%=============================================================================
\maketitle
\indent
{\tiny \textbf{AMS subject classifications.}  \textbf{93A30}, \textbf{65L80}, \textbf{76S05},  \textbf{93B52} }
%
%65M12 = Stability and convergence of numerical methods PDEs
%65L80 = Numerical methods for differential-algebraic equations
%65M60 = Finite element, Rayleigh-Ritz and Galerkin methods for initial value and initial-boundary value problems involving PDEs
%76S05 = Flows in porous media; filtration; seepage
%93A30 = Mathematical modeling (models of systems, model-matching, etc.)
%93B52 = Feedback control
%
%
%=============================================================================
%=========  Introduction
%=============================================================================
\section{Introduction}
The port-Hamiltonian (pH) framework constitutes an energy-based model paradigm that offers a systematic approach for the interactions of (physical) systems with each other and with the environment and thus extends Hamiltonian systems to open physical systems. The pH structure, introduced originally in~\cite{MasS92}, see \cite{RasCSS20,SchJ14} for an overview, provides a geometric description of the model in terms of a Dirac structure, which directly encodes system properties such as passivity and stability into the system of equations. Structure-preserving methods, such as space-discretization~\cite{Egg19,SerMH19}, model-order reduction~\cite{BeaG11,ChaBG16,GugPBS09,PolS11,WolLEK10}, and time-integration~\cite{KotL18,MehM19}, ensure the preservation of these properties throughout numerical approximation schemes. Recently, the definition of pH systems was extended to cover implicit systems~\cite{BeaMXZ18,MehM19,Sch13,SchM18}, resulting in so-called \emph{port-Hamiltonian differential-algebraic equations} (pH-DAEs). Although this system class covers a wide range of equations and simplifies the mathematical theory in many aspects, almost all results are obtained for finite-dimensional DAEs. Extensions of the pH framework to (constrained) infinite-dimensional systems typically do not exist in a general form but rather consider particular applications or special model classes~\cite{AltS17,Egg19,JavZ12,Kot19,MacVM04a,Ram19}.

In view of extending the range of application classes, this paper is devoted to the generation and analysis of the pH structure for partial differential equations which model the deformation of porous media saturated by an incompressible viscous fluid, namely \emph{poroelasticity}~\cite{Biot41,DetC93,EggS19,Sho00}. These equations represent coupled systems of different types of differential equations and, hence, form a so-called \emph{partial differential-algebraic equation} (PDAE) in total~\cite{AltMU20}. The main interest in terms of applications can be found in the field of geomechanics~\cite{Zob10}, often in connection with heterogeneous structures~\cite{AltCMPP20,BroV16a,FuACMPP19}. Moreover, also the displacement of a material due to temperature changes can be modeled with a mathematically equivalent model called thermoelasticity, cf.~\cite{Bio56,MalP17}. 	

The full strength of the pH approach, however, becomes visible in the network case, including multiple pressure variables. Such models typically appear in medical applications such as cerebral infusion tests~\cite{Eis12,SobEWC12} or the investigation of cerebral edema~\cite{VarCTHLTV16}. In both cases, the brain is modeled as a poroelastic medium saturated by different fluids. The formulation as a pH system then allows an interpretation as a network of submodels with internal energies and interconnections.
	
This paper is structured as follows. After introducing the classical two-field formulation of poroelasticity in terms of a weak operator formulation in \Cref{sec:model}, we discuss its extension to the network case. Although the system equations are quasi-static, we include the usually omitted second-order term to find a pH formulation. In \Cref{sec:PH} we recall the finite-dimensional pH-DAE framework and show that the poroelastic equations can be cast in a corresponding infinite-dimensional structure.
In \Cref{sec:interconn} we then derive the pH structure from an interconnection point of view. This means that we already consider the original poroelastic system as a coupled system of pH subsystems. Finally, we summarize the paper in \Cref{sec:conclusion}.
%
%
%=============================================================================
%=========  Poroelastic Network Models
%=============================================================================
\section{Poroelastic Network Models}\label{sec:model}
In this section, we introduce the model equations of poroelasticity and a corresponding operator formulation, which corresponds to the weak form. In view of the subsequent reformulations, it turns out that the full model is better suited than the often considered quasi-static form.
This model is then extended by additional pressure variables leading to the multiple-network case.
%
%
%%%%%%%%%%%%%%%%%%%%%%%%%%%%%%
\subsection{Two-field formulation}\label{sec:model:two}
The original model of linear poroelasticity in a bounded Lipschitz domain $\Omega \subseteq \mathbb{R}^d$ with $d\in\{2,3\}$ (and boundary $\partial \Omega$) was introduced in~\cite{Biot41}, see \cite{Sho00} for a modern formulation covering different modeling components. Within this model, the unknown displacement field and the pressure are averaged quantities across (infinitesimal) cubic elements. Hence, both variables can be treated as variables on the entire domain~$\Omega$. Considering a time horizon~$0<T<\infty$ and setting $\TT \vcentcolon=[0,T]$, one wants to determine the displacement field~$u\colon\TT \times \Omega \rightarrow \mathbb{R}^d$ and the pressure~$p\colon\TT \times \Omega \rightarrow \mathbb{R}$ satisfying the system of coupled partial differential equations
\begin{subequations}
\label{eqn:poro:strong}
\begin{align}
	\rho\, \partial_{tt} u	-\nabla \cdot \big( \sigma (u) \big) + \nabla (\alpha p) &= \hat{f}\phantom{\hat{g}} \quad\ \text{in } (0,T] \times \Omega, \label{eqn:poro:strong:a}\\
	\partial_t\Big(\alpha \nabla \cdot u + \frac{1}{M} p \Big)- \nabla \cdot \Big( \frac{\kappa}{\nu} \nabla p \Big) &= \hat{g}\phantom{\hat{f}} \quad\ \text{in } (0,T] \times \Omega, \label{eqn:poro:strong:b}
\end{align}
which describe the \emph{balance of momentum}~\eqref{eqn:poro:strong:a} and the \emph{conservation of mass} of the fluid~\eqref{eqn:poro:strong:b}. Within this system, the stress tensor $\sigma$ models the linear elastic stress-strain constitutive relation
\begin{equation*}
  \sigma(u) = 2\mu\, \varepsilon (u) + \lambda\, (\nabla \cdot u)\, \mathcal{I}, \qquad
  \varepsilon (u) = \tfrac{1}{2}\, \big(\nabla u + (\nabla u)^T \big)
\end{equation*}
with the Lam\'e coefficients $\mu$ and $\lambda$ and the identity tensor~$\mathcal{I}$.
Further, $\alpha$ denotes the Biot-Willis fluid-solid coupling coefficient, $M$ the Biot modulus, $\kappa$ the permeability, $\rho$ the density, and~$\nu$ the fluid viscosity. The right-hand side~$\hat{g}$ represents an injection or production process and~$\hat{f}$ denotes the volume-distributed external forces. In what follows we will interpret $\hat{g}$ and $\hat{f}$ as distributed inputs to the system. 

The system is equipped with Dirichlet boundary conditions
\begin{align}
	u &= \hat u_b\phantom{p_b} \quad\ \text{on } (0,T] \times \partial \Omega, \label{eqn:poro:strong:c}\\
	p &= \hat p_b\phantom{u_b} \quad\ \text{on } (0,T] \times \partial \Omega\phantom{,} \label{eqn:poro:strong:d}
\end{align}
\end{subequations}
and initial conditions~$p(\,\cdot\,,0) = p^0$, $u(\,\cdot\,,0) = u^0$, as well as~$\partial_t u(\,\cdot\,,0) = \dot{u}^0$
%

% weak/operator form
For the weak formulation of the poroelastic model, we assume homogeneous boundary conditions, i.e., we assume that~$\hat u_b$ and $\hat p_b$ are zero. Inhomogeneous boundary conditions will be briefly discussed in~\Cref{rem:Dirichlet}.
Let us introduce the spaces
\begin{equation*}
  \V\vcentcolon=[H^1_0(\Omega)]^d, \qquad
  \cHV\vcentcolon=[L^2(\Omega)]^d, \qquad
  \Q\vcentcolon=H^1_0(\Omega), \qquad
  \cHQ\vcentcolon=L^2(\Omega),
\end{equation*}
where $L^2(\Omega)$ is the space of square-Lebesgue integrable functions on $\Omega$ and $H^1_0(\Omega)$ is the Hilbert space of functions that have a weak derivative in $L^2(\Omega)$ and satisfy zero boundary conditions on
$\partial \Omega$. Hence, $\V$ and $\Q$ already include the respective homogeneous boundary conditions.
With the corresponding dual spaces~$\V^*$ and~$\Q^*$ we have two Gelfand triples at hand, namely~$\V, \cHV, \V^*$ and~$\Q, \cHQ, \Q^*$. Recall that this means that~$\V$ is continuously and densely embedded in~$\cHV$, which in turn implies a continuous and dense embedding from~$\cHV \cong \cHV^*$ to $\V^*$, cf.~\cite[Ch.~23.4]{Zei90a}. For both triples we will use the notion~$\langle\, \cdot\,,\,\cdot\,\rangle$ for the respective duality pairing.
The resulting embedding is denoted by~$\calI\colon \Q\to\Q^*$, $\langle\calI\, \cdot\,,\,\cdot\,\rangle = (\, \cdot\,,\,\cdot\,)_{\cHQ}$.

% operators
We introduce the operators~$\calY\colon\cHV\to\cHV^*$, $\calM\colon\cHQ\to\cHQ^*$ by
\begin{align*}
\langle\calY u,v\rangle \vcentcolon= \int_\Omega \rho\,  u\, v \dx, \qquad
\langle\calM p,q\rangle \vcentcolon= \int_\Omega \frac{1}{M}\, p\, q \dx
\end{align*}
as well as the differential operators $\calA\colon\V\to\V^*$, $\calK\colon\Q\to\Q^*$, and $\calD\colon\V\to\cHQ^*$ by
\begin{gather*}
  \langle\calA u,v\rangle \vcentcolon= \int_\Omega \sigma(u) : \varepsilon(v) \dx, \quad
  \langle\calK p,q\rangle \vcentcolon= \int_\Omega \frac{\kappa}{\nu}\, \nabla p \cdot \nabla q \dx,\quad
  \langle\calD u,q\rangle \vcentcolon= \int_\Omega \alpha\, (\nabla \cdot u)\, q \dx,
\end{gather*}
where we use for $\calA$ the classical double dot notation from continuum mechanics.
Note that, due to the integration by parts formula, the operator~$\calD$ can also be considered in the form~$\calD\colon\cHV\to\Q^*$. 
To obtain the weak formulation, equation~\eqref{eqn:poro:strong:a} is multiplied by a test function $v\in\V$ and~\eqref{eqn:poro:strong:b} by a test function $q\in\Q$. Integration over the computational domain~$\Omega$ then induces the appearance of the previously defined operators. Accordingly, integrals of the right-hand sides multiplied by a test function arise. For this, we introduce~$f(t) \vcentcolon= \int_\Omega \hat f(t)\, \cdot \dx \in\cHV^*$ and~$g(t) \vcentcolon= \int_\Omega \hat g(t)\, \cdot \dx \in \cHQ^*$.
% comment: rhs in L2 passend zu Existenz-Theorie in [Wloka, Ch. V]  bzw [LioM72 Ch. 3 Sect. 8]

It follows from Korn's inequality~\cite[Th.~6.3.4]{Cia88} that the operator~$\calA$ is elliptic and, thus, invertible. Further, it is self-adjoint and bounded in~$\V$. Similarly, the operator~$\calK$ is self-adjoint, elliptic, and bounded in~$\Q$. The operator~$\calM$ mainly contains the multiplication by the constant~$1/M>0$ and is thus self-adjoint, elliptic, and bounded in the pivot space~$\cHQ$, although it may also be interpreted as an operator~$\calM\colon \Q\to\Q^*$. In the latter setting, however, the operator is only continuous and not elliptic.
Similarly, the operator~$\calY$ represents the multiplication with the constant~$\rho>0$ and thus is self-adjoint, elliptic, and bounded in~$\cHV$.
Finally, the coupling operator~$\calD$ equals, up to the multiplicative prefactor~$\alpha$, the divergence and is thus bounded.

The weak form of~\eqref{eqn:poro:strong} can be stated  as follows: Determine~$u\colon\TT \to \V$ and~$p\colon\TT \to \Q$ such that for almost all $t\in(0,T)$ we have
\begin{subequations}
\label{eqn:twoField:operator}
\begin{alignat}{5}
   \calY \ddot {u}(t)\ &\ &\ +\ &\calA u(t) &\ -\ &\calD^* p(t) &\ =\ &f(t) &&\qquad \text{in } \V^*, \label{eqn:twoField:operator:a}\\
   & \calD \dot{u}(t) &\ +\ & \calM\dot{p}(t) &\ +\ &\calK p(t) &\ =\ &g(t) &&\qquad \text{in } \Q^*  \label{eqn:twoField:operator:b}
\end{alignat}
\end{subequations}
for initial conditions~$p^0\in\cHQ$, $u^0\in\V$, and~$\dot{u}^0 \in \cHV$.
Note that~$\calD^*\colon\cHQ\to\V^*$ denotes the dual operator of~$\calD$, i.e., $\langle \calD^* q, v\rangle = \langle q, \calD v\rangle$ for all~$v\in\V$ and $q\in \Q$.
Assuming right-hand sides~$f \in L^2(\TT;\cHV)$, $g\in L^2(\TT;\cHQ)$, suitable solution spaces (for the weak formulation) read~$u\in L^2(\TT;\V)$, $p\in L^2(\TT;\Q)$ with~$\dot{u}\in L^2(\TT;\cHV)$, $\ddot{u}\in L^2(\TT;\V^*)$, and~$\dot{p}\in L^2(\TT;\Q^*)$.
Sobolev embeddings and the existence theory of~\cite[Ch.~3, Sect.~8.4]{LioM72} then imply that~$p\in \calC(\TT;\cHQ)$ and~$u\in \calC(\TT;\V)$, $\dot{u}\in \calC(\TT;\cHV)$ such that initial conditions are meaningful.

Within this paper, we will study system~\eqref{eqn:twoField:operator} in the corresponding formulation in terms of operator matrices, i.e.,
\begin{align}
\label{eqn:twoField:opMatrix}
\begin{bmatrix} \calY & 0\\ 0 & 0 \end{bmatrix}
  \begin{bmatrix} \ddot{u}\\ \ddot{p} \end{bmatrix}+
  \begin{bmatrix} 0 & 0\\ \calD & \calM \end{bmatrix}
  \begin{bmatrix} \dot{u}\\ \dot{p} \end{bmatrix}
 +
  \begin{bmatrix} \calA & -\calD^*\\ 0 & \calK \end{bmatrix}
  \begin{bmatrix} u\\ p \end{bmatrix}
  = \begin{bmatrix} f\\ g \end{bmatrix}.
\end{align}
Besides the second-order formulation presented here, we will study an associated first-order formulation by introducing a new variable for $\dot u$. This will be investigated together with the pH formulation of the model equations in \Cref{sec:PH:twoField}. As a result, we study~\eqref{eqn:twoField:opMatrix} as two systems that are interconnected via power-conserving feedback in \Cref{sec:interconn}, giving a system-theoretic interpretation to these equations. 
\begin{remark}
	\label{rem:poroNonlinear}
	In several applications, see for instance \cite{Car39}, the permeability~$\kappa$, or more precisely, the hydraulic conductivity, depends on the dilatation~$\nabla\cdot u$, i.e., $\kappa = \kappa(\nabla\cdot u)$. In this case, the term~$\calK p$ becomes nonlinear. Assuming that there exist positive constants $\kappa_-,\kappa_+$ satisfying $0 < \kappa_- \leq \kappa(\xi) \leq \kappa_+$ for all $\xi\in\R$ still renders the operator $\calK(u)$ self-adjoint and elliptic, which is sufficient for the forthcoming analysis. For instance, this is the case for the Kozney-Carmen type hydraulic conductivity~\cite{CaoCM13,Car39}, see also~\cite{AltMU20b} for a corresponding numerical setup.
\end{remark}
%
%
%%%%%%%%%%%%%%%%%%%%%%%%%%%%%%
\subsection{Quasi-static formulation}\label{sec:model:twoQS}
In many applications, the coefficient $\rho$ in \eqref{eqn:poro:strong:a} is considered small enough so that the second order term $\rho \ddot u$ can  be neglected. This  leads to the so-called \emph{quasi-static formulation}, studied, for instance, in~\cite{Egg19,Sho00}. Note that a detailed asymptotic analysis considering different orders in small terms may lead to a different formulation than just formally considering $\rho \ddot u = 0$. 

In view of the operator formulation~\eqref{eqn:twoField:operator}, the quasi-static model is given by
\begin{subequations}
	\label{eqn:twoFieldQS}
	\begin{align}
		\label{eqn:twoFieldQS:elliptic}\calA u - \calD^* p &= f \quad\ \text{in } \V^*,\\
		\label{eqn:twoFieldQS:parabolic}\calD \dot{u} + \calM \dot{p} + \calK p &= g \quad\ \text{in } \Q^*.
	\end{align}
\end{subequations}	
\begin{remark}
	\label{rem:consistentInitialValue}
	For the quasi-static formulation~\eqref{eqn:twoFieldQS} there is no need to prescribe an initial value for $\dot{u}$. In fact, it it not even necessary to prescribe an initial value for $u$, since~$u^0$ is uniquely determined from~\eqref{eqn:twoFieldQS:elliptic}, see for instance~\cite{AltMU20} and the forthcoming \Cref{sec:PH:DAE}.
	The dependence of~$u^0$ on~$p^0$ is a typical \emph{consistency condition} known for PDAEs.
\end{remark}
If $f$ is differentiable in time, then, since the operator $\calA$ is invertible, we can formally solve~\eqref{eqn:twoFieldQS:elliptic} for $u$ and insert the result into~\eqref{eqn:twoFieldQS:parabolic}. Then one obtains a parabolic equation for the pressure $p$ given by
\begin{equation}
	\label{parabolic}
	\widetilde{\calM} \dot{p}
%  \vcentcolon= (\calM + \calD \calA^{-1}\calD^*)\, \dot{p}
+ \calK p
  =  \tilde{g},
\end{equation}
with~$\widetilde{\calM} \vcentcolon= \calM + \calD \calA^{-1}\calD^*\colon \cHQ\to\cHQ^*$ and adapted right-hand side~$\tilde{g}\vcentcolon=g-\calD \calA^{-1}\dot{f}$.
For any computed value of $p(t)$ at time $t$, the solution $u(t)$ can be computed by solving the elliptic system
\begin{displaymath}
	\calA u(t) =\calD^* p(t) + f(t).
\end{displaymath}
%

% thermo
We would like to emphasize that system~\eqref{eqn:twoFieldQS} covers general elliptic-parabolic problems~\cite{AltMU20}. Clearly, this includes (quasi-static) linear poroelasticity and thermoelasticity. In contrast, the full model~\eqref{eqn:twoField:operator} covers a general hyperbolic-parabolic coupling.
%
%%%%%%%%%%%%%%%%%%%%%%%%%%%%%%
\subsection{Multi-field network systems}\label{sec:model:network}
For medical applications, the poroelastic model~\eqref{eqn:poro:strong} is often extended by additional pressure variables, e.g., to distinguish vessel types in the investigation of cerebral edema, cf.~\cite{VarCTHLTV16}. For instance, one may distinguish arterial and venous blood flow, which we schematically depict in \Cref{fig:mpet}. Here, the arterial and venous blood each have their own (averaged) pressure variables. For a general description of multiple-network poroelastic theory we refer to \cite{TulV11}. 
\begin{figure}
	\centering
	\begin{tikzpicture}
		\clip (0,0) rectangle (7,4);
		\fill[black!15] (0,0) rectangle (7,4);
		\node[anchor=north east] at (6.8,3.8) {\text{tissue}};
		
		%\draw (0,0) grid (7,4);
		
		\draw[line width=4pt,red] (-0.1,2) ..  controls (3,4) and (3.5,0) .. (7.1,1.5);
		\draw[line width=4pt,blue] (-0.1,0.3) ..  controls (3,4) and (3.5,-0.2) .. (5,4.1);
		\draw[line width=3pt,blue] (2.5,2) .. controls (4,-1.2) and (5,4) .. (7.1,2.9);
		\draw[line width=3pt,red] (3,2.1) .. controls (4,4) and (5,1) .. (7.1,3.5);
		\draw[line width=3pt,red] (5.9,2.55) .. controls (6.5,2.6) .. (7.1,1.9);
		\draw[line width=3pt,blue] (4.8,1.95) .. controls (6.1,1.6) .. (7.1,2.3);
		\draw[line width=3pt,red] (1,2.5) .. controls (2,0) .. (4,-0.1);
		\draw[line width=3pt,blue] (1,1.4) .. controls (2,4) .. (3,4.1);
		\draw[line width=2pt,blue] (0.6,1) .. controls (1.5,0.2) and (2.5,2) .. (4.5,-0.1);
		\draw[line width=2pt,red] (1.4,1.6) .. controls (0.7,0.7) .. (0.5,-0.1);
		\draw[line width=2pt,red] (2,2.5) .. controls (2.5,3) .. (2.5,4.1);
		\draw[line width=2pt,red] (2.3,2.8) .. controls (1,3) .. (0.7,4.1);
		\draw[line width=2pt,red] (2,0.45) .. controls (3,2) and (5,0) .. (7.1,0.5);
		\draw[line width=1pt,red] (4.5,2.6) .. controls (5.1,2.5) and (5.3,4) .. (5.8,4.1);
		\draw[line width=1pt,red] (5.05,3) .. controls (4.5,4) and (4,3) .. (3.5,4.1);		
	\end{tikzpicture}
	\caption{Schematic illustration of arterial (red) and venous (blue) blood flow networks in the tissue (gray) acting as the porous material}
	\label{fig:mpet}
\end{figure}
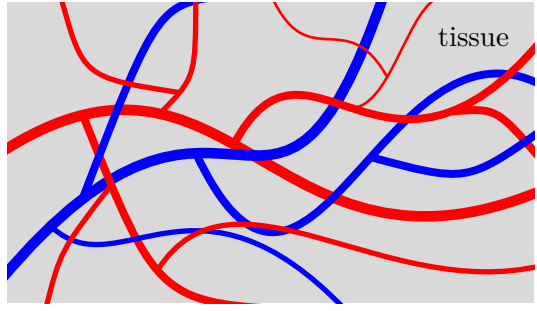

As before, such applications usually come with non-zero boundary conditions, cf.~\Cref{rem:Dirichlet} below. For our analysis, however, we restrict ourselves to homogeneous Dirichlet boundary conditions to focus on the structure of the equations.

In the multi-field network case we want to compute the displacement $u\colon\TT \times\Omega\to \R$ and pressures~$p_i\colon\TT \times\Omega \to \R$, $i=1,\dots,m$, such that
\begin{subequations}
\label{eqn:network:strong}
\begin{align}
 \rho\, \partial_{tt} u -\nabla \cdot \big( \sigma (u) \big) + \sum_{i=1}^m \nabla (\alpha_i p_i) &= \hat{f} && \text{in } (0,T] \times \Omega, \\
  \partial_t\Big(\alpha_i \nabla \cdot u + \frac{1}{M} p_i \Big)- \nabla \cdot \Big( \frac{\kappa_i}{\nu_i} \nabla p_i \Big) - \sum_{j\neq i} \beta_{ij} (p_i-p_j) &= \hat{g}_i && \text{in } (0,T] \times \Omega
\end{align}
\end{subequations}
with initial conditions for $u$, $\dot{u}$, and all $p_i$. Note that the second equation has to be considered for $i=1,\dots, m$, leading to~$m+1$ equations in total. The parameters~$\beta_{ij}$ model exchange rates from one subsystem to the other and are usually small. In several applications, symmetric exchange rates~$\beta_{ij}=\beta_{ji}$ are considered, see e.g.~\cite{TulV11}.

To shorten notation, we introduce the matrix~$B \vcentcolon= [\beta_{ij}]_{i,j=1}^m \in \R^{m,m}$ with diagonal entries~$\beta_{ii} \vcentcolon= -\sum_{j\neq i} \beta_{ij}$, $i=1,\dots,m$.
Thus, $B$ contains all exchange rates of the system in the off-diagonal elements and the row sums are all equal to zero by construction. In this way, the exchange conditions resemble Kirchhoff's laws in electric or energy transport networks \cite{EggKLMM18,Sch13,SchJ14}.

For the weak formulation of~\eqref{eqn:network:strong} we define operators~$\calK_i\colon\Q\to\Q^*$ and~$\calD_i\colon\V\to\cHQ^*$ as in \Cref{sec:model:two} with the same properties as~$\calK$ and~$\calD$, respectively. We further introduce the block operators
\begin{equation*}
  \calKK{B} \vcentcolon=
  \begin{bmatrix} \calK_1 &  & \\  & \ddots & \\ & & \calK_m \end{bmatrix} + B \otimes \calI, \qquad
  \calMM \vcentcolon=
  \begin{bmatrix} \calM &  & \\  & \ddots & \\ & & \calM \end{bmatrix}, \qquad
  \calDD \vcentcolon= \begin{bmatrix} \calD_1 \\  \vdots \\ \calD_m \end{bmatrix}.
\end{equation*}
Note that, in order to allow a compact notion of the network case, we write~$\Q^m$ and~$\cHQ^m$ for the product spaces~$\Q\times\dots\times\Q$ and~$\cHQ\times\dots\times \cHQ$, respectively.
Further, $B \otimes \calI$ mimics the Kronecker product, i.e., for~$\mathbf{p} = [p_1,\dots,p_m]^T, \mathbf{q} = [q_1,\dots,q_m]^T\in\Q^m$ we have
\begin{displaymath}
  \langle (B \otimes \calI) \mathbf p, \mathbf q\rangle
  \vcentcolon= \sum_{i=1}^m \sum_{j=1}^m \beta_{ij}\, \langle \calI p_j, q_i\rangle
  = \sum_{i=1}^m \sum_{j=1}^m \beta_{ij}\, (p_j, q_i)_{\cHQ}.
\end{displaymath}
\begin{lemma}
	\label{lem:ellipticBB}
	The bounded operator~$\calKK{B}\colon \Q^m \to [\Q^m]^*$ is elliptic for sufficiently small exchange rates~$\beta_{ij}$. Further, $\calKK{B}$ is self-adjoint if the matrix~$B$ is symmetric.
\end{lemma}
\begin{proof}
The claim on the symmetry of~$\calKK{B}$ is obvious, since all operators~$\calK_i$ are self-adjoint. Let~$c$ denote the (uniform) ellipticity constant for all~$\calK_i$ and~$C_\beta\vcentcolon= \max_{i\neq j} |\beta_{ij}|$. Then,  for~$\mathbf{p} = [p_1,\dots,p_m]^T\in\Q^m$ we obtain the estimate
\begin{displaymath}
  \langle \calKK{B} \mathbf{p}, \mathbf{p}\rangle
  = \sum_{i=1}^m \langle \calK_i p_i, p_i\rangle + \langle (B \otimes \calI) \mathbf{p}, \mathbf{p}\rangle
  \ge c\sum_{i=1}^m \|p_i\|^2_\Q - 2m\, C_\beta\sum_{i=1}^m \|p_i\|^2_{\cHQ}.
\end{displaymath}
The continuous embedding $\Q\hookrightarrow\cHQ$ implies~$\langle \calKK{B} \mathbf{p}, \mathbf{p}\rangle \ge (c-2m\, C_\beta C_{\Q\hookrightarrow\cHQ}^2) \|\mathbf{p}\|^2_\Q$ and hence the ellipticity of~$\calKK{B}$ for~$C_\beta < c/(2m\, C_{\Q\hookrightarrow\cHQ}^2)$.
Here $C_{\Q\hookrightarrow\cHQ}$ denotes the constant of the continuous embedding~$\Q\hookrightarrow\cHQ$.
\end{proof}
\begin{assumption}
	\label{ass:smallExchangeRates}
	Throughout this paper, we assume that the operator~$\calKK{B}$ is elliptic and thus invertible.
\end{assumption}		
Note that \Cref{ass:smallExchangeRates} does {\em not} include the symmetry of~$B$. For the nonsymmetric case we make the following observation.
\begin{lemma}
	\label{lem:smallExchangeRatesSymBeta}
	Let $B_{\mathrm{sym}} \vcentcolon= \tfrac{1}{2}(B+B^T)$. Then $\calKK{B}$ is elliptic if and only if $\calKK{B_{\mathrm{sym}}}$ is elliptic.
\end{lemma}
\begin{proof}
Let $\mathbf{q}\in\Q^m$. The proof follows immediately from the identity
\begin{displaymath}
	\langle \calKK{B}\mathbf{q},\mathbf{q}\rangle
%	= \langle \calKK{B_{\mathrm{sym}}}\mathbf{q},\mathbf{q}\rangle + \tfrac{1}{2} \big\langle ((B-B^T)\otimes \calI) \mathbf{q},\mathbf{q} \big\rangle
	= \langle \calKK{B_{\mathrm{sym}}}\mathbf{q},\mathbf{q}\rangle + \tfrac{1}{2} \big\langle (B\otimes \calI-B^T\otimes \calI) \mathbf{q},\mathbf{q} \big\rangle
\end{displaymath}
and the observation that
\begin{displaymath}
	\langle (B\otimes \calI-B^T\otimes \calI)\mathbf{q},\mathbf{q}\rangle
	= \langle (B\otimes \calI)\mathbf{q},\mathbf{q}\rangle - \langle (B\otimes \calI)\mathbf{q},\mathbf{q}\rangle
	= 0.\qedhere
\end{displaymath}
\end{proof}

Let us emphasize that \Cref{lem:ellipticBB} indicates that small exchange rates in the sense that
\begin{equation*}
  \max_{i\neq j} |\beta_{ij}|
  < \frac{c}{2m\, C_{\Q\hookrightarrow\cHQ}^2},
\end{equation*}
with~$c$ denoting the uniform ellipticity constant of all~$\calK_i$, is a sufficient condition for the ellipticity in \Cref{ass:smallExchangeRates}.

The weak formulation of~\eqref{eqn:network:strong} in terms of the introduced operator matrices has the form: Determine~$u\colon\TT \to \V$ and~$\mathbf p\colon\TT \to \Q^m$ such that for almost all $t\in \TT$ the operator equation
\begin{align}
\label{eqn:network:opMatrix}
\begin{bmatrix} \calY & 0 \\ 0 & 0 \end{bmatrix}  \begin{bmatrix} \ddot{u}\\ \ddot{\mathbf p} \end{bmatrix}+ \begin{bmatrix} 0 & 0 \\ \calDD & \calMM \end{bmatrix}
  \begin{bmatrix} \dot{u}\\ \dot{\mathbf p} \end{bmatrix}
  +
  \begin{bmatrix} \calA & -\calDD^* \\ 0 & \calKK{B} \end{bmatrix}
  \begin{bmatrix} u\\ \mathbf p \end{bmatrix}
  = \begin{bmatrix} f\\ {\mathbf g} \end{bmatrix}
\end{align}
is satisfied. Here we have used the notion~$\mathbf g \vcentcolon= [g_1,\dots,g_m]^T$, where the~$g_i$ are defined in terms of~$\hat g_i$ as in \Cref{sec:model:two}.

Thus, the network case has the same structure as the two-field poroelastic model, cf.~system~\eqref{eqn:twoField:opMatrix}. A direct consequence is that a semi-discretization in space of the multi-field network model~\eqref{eqn:network:opMatrix} yields an operator equation of the same structure as the two-field model.
%
%%%%%%%%%%%%%%%%%%%%%%%%%%%%%%
\subsection{Quasi-static network case}\label{sec:model:networkQS}
Similar to the two-field case considered above, we briefly discuss the case where the coefficient~$\rho$ is assumed to be small and thus neglected. This then yields the quasi-static network model
\begin{align*}
%\label{eqn:networkQS}
\begin{bmatrix} 0 & 0 \\ \calDD & \calMM \end{bmatrix}
\begin{bmatrix} \dot{u}\\ \dot{\mathbf p} \end{bmatrix}
+
\begin{bmatrix} \calA & -\calDD^* \\ 0 & \calKK{B} \end{bmatrix}
\begin{bmatrix} u\\ \mathbf{p} \end{bmatrix}
= \begin{bmatrix} f\\ {\mathbf g} \end{bmatrix}
\end{align*}
with initial condition~$\mathbf{p}(0) = \mathbf{p}^0 \in \cHQ^m$.
As in the two-field model, a consistency condition then uniquely determines~$u^0$.

Assuming once more that~$f$ is differentiable in time, we can use the invertibility of~$\calA$ to eliminate the first equation. This then leads to a system of~$m$ coupled parabolic equations, namely
\begin{equation*}
\widetilde{\calM} \dot{\mathbf p}
+ \calKK{B} \mathbf p
=  \tilde{\mathbf g}
\end{equation*}
with~$\widetilde{\calM} \vcentcolon= \calMM + \calDD \calA^{-1}\calDD^*\colon \cHQ^m\to[\cHQ^m]^*$ and adapted right-hand side~$\tilde{\mathbf g}\vcentcolon=\mathbf g-\calDD \calA^{-1}\dot{f}$.
To restore the displacement~$u$, one computes the solution to the elliptic problem~$\calA u =\calDD^* \mathbf{p} + f$.
\begin{example}[Poroelastic brain model]
The example in~\cite{VarCTHLTV16} considers $\rho \ddot u=0$, $m=4$ pressure variables, and no external forces or injections, i.e., $f\equiv 0$ and $g_i\equiv 0$. This means that the system is driven only by its hydrostatic pressure gradients. 
\end{example}
%
%
%=============================================================================
%=========  pH models
%=============================================================================
\section{Port-Hamiltonian descriptor system formulation}\label{sec:PH}
To render the paper self-contained, we recall the definition and some results for pH-DAEs, also called {\em port-Hamiltonian descriptor systems}. For the finite-dimensional setting, this is presented in \cite{MehM19}, see also earlier versions in~\cite{BeaMXZ18,Sch13}.
Afterwards, we show that the poroelastic equations can be formulated in a corresponding way, leading to a port-Hamiltonian partial differential-algebraic equation (pH-PDAE).
%
%
%%%%%%%%%%%%%%%%%%%%%%%%%%%%%%
\subsection{Port-Hamiltonian framework}\label{sec:PH:ODEs}
We start with the definition of a pH-DAE in the finite-dimensional framework, which we will mimic in the pH-PDAE setting. We present a special case of the definition given in~\cite{BeaMXZ18} % [namely $Q = \Id$]
that is sufficient for our analysis, using a slightly different notation than therein. For a generalization to time-dependent and nonlinear descriptor systems we refer to \cite{MehM19}.
\begin{definition}[pH-DAE]
	\label{def:pH-DAE}
	A descriptor system of the form
	\begin{equation}
		\label{eqn:pH-DAE}
		\begin{bmatrix}
			E\dot{z}\\
			y
		\end{bmatrix} = \begin{bmatrix}
			A & B\\
			C & D
		\end{bmatrix}\begin{bmatrix}
			z\\
			v
		\end{bmatrix}
	\end{equation}
	with \emph{state} $z\colon \mathbb{T}\to\R^n$, \emph{input} $v\colon \mathbb{T}\to\R^m$, and \emph{output} $y\colon \mathbb{T}\to \R^m$ is called \emph{pH-DAE} with associated quadratic \emph{Hamiltonian}
	\begin{displaymath}
		\mathcal{H}\colon \R^n \to \R,\qquad
		\mathcal{H}(z)
		\vcentcolon= \langle  E z,z\rangle
		= z^T E z
	\end{displaymath}
	if $E^T = E \in \R^{n\times n}$ is positive semi-definite and the symmetric matrix
	\begin{displaymath}
		W
		\vcentcolon= \mathrm{sym}\left(\begin{bmatrix} -A & -B\\
		C & D
		\end{bmatrix}\right)
		\vcentcolon= \frac{1}{2}\left(\begin{bmatrix}
			-A & -B\\
			C & D
		\end{bmatrix} + \begin{bmatrix}
			-A & -B\\
			C & D
		\end{bmatrix}^T\right) \in \R^{(n+m)\times (n+m)},
	\end{displaymath}
	called the \emph{dissipation matrix}, is positive semi-definite.
\end{definition}
Note that~\eqref{eqn:pH-DAE} can be equivalently written as
\begin{subequations}
	\label{eqn:pH-DAEv2}
	\begin{align}
			E\dot{z} &= (J-R)z + (G-P)v,\\
			y &= (G+P)^Tz + (S-N)v,		
		\end{align}
\end{subequations}
with skew-symmetric matrices $J \vcentcolon= \tfrac{1}{2}(A-A^T)$ and $N\vcentcolon= -\tfrac{1}{2}(D - D^T)$, symmetric matrices $R \vcentcolon=  -\tfrac{1}{2}(A+A^T)$ and $S \vcentcolon= \tfrac{1}{2}(D+D^T)$, and $G \vcentcolon= \tfrac{1}{2}(B - C^T)$, $P \vcentcolon= -\tfrac{1}{2}(B + C^T)$. Computing the dissipation matrix for~\eqref{eqn:pH-DAEv2} yields
\begin{equation}
	\label{eqn:dissipationMatrix}
	W = \begin{bmatrix}
		R & P\\
		P^T & N
	\end{bmatrix}.
\end{equation}
We thus obtain the following equivalent formulation of a pH-DAE.
\begin{lemma}
	\label{lem:pHequivalent}
	The descriptor system~\eqref{eqn:pH-DAEv2} with symmetric matrices $R = R^T$, $S = S^T$, skew-symmetric matrices $J=-J^T$, $N=-N^T$, and symmetric positive semi-definite matrix $E = E^T$ is a pH-DAE, if the dissipation matrix~\eqref{eqn:dissipationMatrix} is positive semi-definite.
\end{lemma}
Let us emphasize that if no feedthrough term is present, i.e., $D=0$ in \eqref{eqn:pH-DAE} or $S=N=0$ in \eqref{eqn:pH-DAEv2}, then the dissipation matrix~\eqref{eqn:dissipationMatrix} being semi-definite immediately implies $P=0$ such that the pH formulation is given as
\begin{subequations}
	\label{eqn:pH-DAEv3}
	\begin{align}
			E\dot{z} &= (J-R)z + Gv,\\
			y &= G^Tz,
		\end{align}
\end{subequations}
with skew-symmetric $J$ and symmetric positive semi-definite $E$ and $R$. For our remaining analysis, the latter formulation is sufficient, serving as the main reason to restrict ourselves to systems of the form~\eqref{eqn:pH-DAEv3}.

The class of pH-DAEs, respectively the extensions presented in \cite{MehM19}, have many nice properties, see also~\cite{BeaMXZ18}. The general definition in \cite{MehM19} extends to  weak solutions and infinite dimension, the structure is invariant under state-time diffeomorphisms, the structure is invariant when making the system autonomous, and the class can be described via a Dirac structure. Furthermore, pH-DAEs satisfy a power balance equation and dissipation inequality.
\begin{theorem}[Dissipation inequality for pH-DAEs~\cite{BeaMXZ18,MehM19}]
\label{thm:di}
Consider a pH-DAE of the form~\eqref{eqn:pH-DAE}. Then the power balance equation
  \begin{equation}\label{eq:powerBalanceEq}
    \ddt \mathcal{H}(z(t)) = - \begin{bmatrix} z \\ v\end{bmatrix}^T W \begin{bmatrix}z \\ v \end{bmatrix}+ y^Tv,
  \end{equation}
holds along any solution $z$, for any input $v$. In particular, the \emph{dissipation inequality}
  \begin{equation}\label{eq:dissIneq}
    \mathcal{H}(z(t_2)) - \mathcal{H}(z(t_1)) \leq \int_{t_1}^{t_2}y(\tau)^Tv(\tau) \,\mathrm{d}\tau
  \end{equation}
holds for all $t_2 > t_1$.
\end{theorem}
Another important property, see \cite{MehM19}, is that the class of pH-DAEs is invariant under power-conserving or dissipative interconnection. To see this, consider two pH-DAEs
\begin{align*}
  E_i\dot{z}_i &= (J_i- R_i)z_i + G_iv_i, \\
  y_i &= G_i^Tz_i
\end{align*}
with Hamiltonians $\mathcal{H}_i\colon\mathbb{R}^{n_i}\to\mathbb{R}$, for $i=1,2$, and assume that the two systems are interconnected via
\begin{displaymath}
	\begin{bmatrix}
		v_1\\
		v_2
	\end{bmatrix} = \begin{bmatrix}
		F_{11} & F_{12}\\
		F_{21} & F_{22}\end{bmatrix}
		\begin{bmatrix}
		y_1\\
		y_2
	\end{bmatrix} + \begin{bmatrix}
		\tilde{v}_1\\
		\tilde{v}_2
	\end{bmatrix},\qquad F \vcentcolon= \begin{bmatrix}
		F_{11} & F_{12}\\
		F_{21} & F_{22}\end{bmatrix}.
\end{displaymath}
Then, defining the aggregated state, input, and output, respectively, via
\begin{displaymath}
  z \vcentcolon= \begin{bmatrix} z_1 \\ z_2 \end{bmatrix}, \qquad
  \tilde{v} \vcentcolon= \begin{bmatrix} \tilde{v}_1 \\ \tilde{v}_2 \end{bmatrix}, \qquad
  y \vcentcolon= \begin{bmatrix} y_1 \\ y_2 \end{bmatrix},
\end{displaymath}
and aggregated system matrices $J \vcentcolon= \mathrm{diag}(J_1,J_2)$, $R \vcentcolon= \mathrm{diag}(R_1,R_2)$, $G \vcentcolon= \mathrm{diag}(G_1,G_2)$, the coupled system is given by
\begin{subequations}
	\label{eqn:interconnectedSystem}
	\begin{align}
		E\dot{z} &= \big(J-R+GFG^T\big)z + G\tilde{v},\\
		y &= G^T\state.
	\end{align}
\end{subequations}
Writing $F = F_{\mathrm{skew}} + F_{\mathrm{sym}}$ with skew-symmetric part $F_{\mathrm{skew}}$ (corresponding to the power-preserving component) and symmetric part $F_{\mathrm{sym}}$, we immediately observe that~\eqref{eqn:interconnectedSystem}
is again a pH system if, and only if, $R-GF_{\mathrm{sym}}G^T$ is positive semi-definite. A sufficient condition to retain the pH structure is thus to require that $F_{\mathrm{sym}}$ is negative semi-definite corresponding to a potentially dissipative component of the interconnection.

Finally, a key property for our later analysis is that the pH-DAE structure is invariant under Galerkin projection in the sense that whenever we use a variational formulation and restrict the dynamics by using the same  subspace as ansatz and test space; cf.~\cite{BeaMG19,Egg19,EggKLMM18}. This will form the basis for our forthcoming analysis, where we show that the poroelastic (network) equations introduced in \Cref{sec:model} exhibit such a pH structure. Since we analyze the equations in the operator formulation, which we will then call a pH-PDAE, we cannot follow the exact wording but only the spirit of \Cref{def:pH-DAE}, respectively \Cref{lem:pHequivalent}. Nevertheless, applying a spatial discretization by a finite element method afterwards will retain the structural properties such that the semi-discretized models are then pH-DAEs.
Since the operator matrices are constant in time and there is no feed-through term, we consider the special case that the matrices~$P$, $S$, and~$N$ are zero. As a result, the requirements of \Cref{def:pH-DAE}, respectively \Cref{lem:pHequivalent}, reduce to $J$ being skew-symmetric and $E$, $R$ being symmetric as well as positive semi-definite.
%
%
%%%%%%%%%%%%%%%%%%%%%%%%%%%%%%
\subsection{PH formulation of the two-field model}\label{sec:PH:twoField}
In this section, we discuss a reformulation of the operator form of the poroelastic equations, which although we discuss the infinite dimensional case exhibits the pH structure in the spirit of \Cref{def:pH-DAE}. Consequently, a Galerkin projection of the operator form will then lead to a (finite-dimensional) pH-DAE. Note furthermore, that the inhomogeneities $f,g$ are domain-distributed inputs, which demands for the infinite dimensional setting.

Starting with the second-order formulation~\eqref{eqn:twoField:opMatrix}, we first perform a first-order (mixed) formulation by introducing $w=\dot{u}$ as new variable and adding the equation~$\calA w= \calA \dot{u}$ to the system. Recall that in our case $\calA^*=\calA$ is elliptic and thus, invertible. Assuming~$\rho>0$, we obtain the (implicit) first-order system
\begin{align}
\label{eqn:twoField:opMatrix3}
\begin{bmatrix}\calY & 0 & 0\\ 0 & \calA & 0 \\ 0 & 0 & \calM\end{bmatrix}
  \begin{bmatrix} \dot{w} \\\dot{u}\\ \dot{p} \end{bmatrix}
  = \begin{bmatrix} 0 & -\calA & \calD^* \\ \calA & 0 & 0 \\ -\calD & 0 & -\calK \end{bmatrix}
  \begin{bmatrix} w\\ u\\ p \end{bmatrix}  + \begin{bmatrix} f\\ 0 \\ g \end{bmatrix}.
\end{align}
Note that the three operators~$\calY$, $\calA$, and $\calM$ appearing on the left-hand side are invertible. Hence, there are no algebraic constraints. The canonical Hamiltonian (energy function) associated with this system is given by
\begin{equation}
	\label{eqn:energyfull}
	\mathcal{H}(w,u,p)
	\vcentcolon= \frac{1}{2}\, \Big(\langle \calY w, w\rangle + \langle \calA u, u\rangle + \langle\calM p, p\rangle\Big),
\end{equation}
where $\tfrac{1}{2}\langle \calY w,w\rangle$ describes the kinetic part of the energy and $\tfrac{1}{2} \langle \calA u, u\rangle + \tfrac{1}{2}\langle \calM p,p \rangle$ the potential energy.
Let us emphasize that although, formally, the two-field model does not feature any control inputs,
we will view the source terms $f$ and $g$ as external (distributed) inputs (respectively forces, cf.~\cite{Sho00})
 to the system and complement the system with the power-conjugated output, i.e., with
\begin{equation}
	\label{eqn:twoField:powerConjugateOutput}
	y = \begin{bmatrix}
		\Id & 0 & 0\\
		0 & 0 & \Id
	\end{bmatrix}\begin{bmatrix}
		w\\u\\p
	\end{bmatrix} = \begin{bmatrix}
		w\\p
	\end{bmatrix}.
\end{equation}
We immediately observe that the first-order formulation~\eqref{eqn:twoField:opMatrix3} together with the power-conjugated output equation~\eqref{eqn:twoField:powerConjugateOutput} emulates the pH-DAE structure with~$\mathcal{P}=0$, $\mathcal{S} = \mathcal{N} = 0$,
\begin{displaymath}
	\mathcal{E} = \begin{bmatrix}
		\calY & 0 & 0\\
		0 & \calA & 0\\
		0 & 0 & \calM
	\end{bmatrix},\quad
	\mathcal{J} = \begin{bmatrix}
		0 & -\calA & \calD^* \\
		\calA & 0 & 0 \\
		-\calD & 0 & 0
	\end{bmatrix},\quad
	\mathcal{R} = \begin{bmatrix}
		0 & 0 & 0 \\
		0 & 0 & 0 \\
		0 & 0 & \calK
	\end{bmatrix},\quad
	\mathcal{B} = \begin{bmatrix}
		\Id & 0 \\
		0 & 0\\
		0 & \Id
	\end{bmatrix}.
\end{displaymath}
Recall that the operators $\calA$, $\calM$, $\calK$, and~$\calY$ are positive definite.
We expect, similarly to the finite-dimensional case considered in~\Cref{thm:di},
a power balance equation and a dissipation inequality. Let us demonstrate this as
an example for the formulation \eqref{eqn:twoField:opMatrix3}. In the other formulations that we discuss in the following,
the result and proof is analogous.
\begin{theorem}[Dissipation inequality]
	Let $(w,u,p)$ satisfy~\eqref{eqn:twoField:opMatrix3}. Then the Hamiltonian defined in~\eqref{eqn:energyfull} satisfies the power balance equation
	\begin{equation}
		\label{eqn:powerBalanceEquation}
		\ddt {\mathcal{H}}(w,u,p) = - \langle \calK p, p\rangle + \langle v, y\rangle % \langle f,u\rangle + \langle g,q\rangle.
	\end{equation}
	In particular, the Hamiltonian satisfies the dissipation inequality
	\begin{equation*}
%		\label{eqn:dissipationInequality}
		\ddt {\mathcal H}(u,q,p) \leq \langle v, y\rangle. %\langle f, u\rangle + \langle g,q\rangle
	\end{equation*}
\end{theorem}
\begin{proof}
	For $v = [f,\, g]^T$ let $(w,u,p,y)$ satisfy~\eqref{eqn:twoField:opMatrix3} and~\eqref{eqn:twoField:powerConjugateOutput}. Then,
\begin{align*}
		\ddt \mathcal{H}(u,q,p) &= \langle \calY\dot{w},w\rangle + \langle \calA \dot{u},u\rangle + \langle \calM \dot{p},p\rangle\\
		&= \langle -\calA u + \calD^*p + f,w\rangle + \langle \calA w,u\rangle - \langle \calD w + \calK p - g,p\rangle\\
		&= -\langle \calK p,p\rangle + \langle v,y\rangle,
\end{align*}
where the last equality follows from $\langle v,y\rangle = \langle f,w\rangle + \langle g,p\rangle$.
\end{proof}
Note that the pH structure also allows the operators~$\calM$ and $\calY$ to become singular, but still semi-definite. In particular, we immediately have the structure also in the (formal) limiting situation $\rho \ddot =0$. Furthermore, the pH-PDAE structure still holds if $\calK$ is not self-adjoint, it suffices that the symmetric part of $\calK$ is positive semi-definite. In addition, we immediately observe that the power balance equation~\eqref{eqn:powerBalanceEquation} is still satisfied for a nonlinear operator~$\calK$ as discussed in \Cref{rem:poroNonlinear}.

\begin{remark}
Due to the inclusion of the equation~$\calA w= \calA \dot{u}$, where we applied the operator~$\calA$ in order to obtain the skew-adjoint structure, the pH-PDAE~\eqref{eqn:twoField:opMatrix3} formally calls for~$\dot{u}\in L^2(\TT;\V)$. Hence, this formulation is only valid for more regular right-hand sides~$f$ and~$g$.
For an alternative formulation, we use the fact that the square root $\calA^{1/2}$ of $\calA$ is well-defined, cf.~\cite[Cha.\,1,\,\S\,3,~Th.~3.35]{Kat95}). Moreover, $\calA^{1/2}$ is self-adjoint and elliptic. Introducing the new variable~$\widetilde{u} \vcentcolon= \calA^{1/2} u$, and applying~$\calA^{-1/2}$ to the second equation, we obtain
\begin{subequations}
\label{eqn:twoField:opMatrix4}
\begin{align}
	\label{eqn:twoField:opMatrix4:state}
	\begin{bmatrix}
		\calY & 0 & 0\\
		0 & \Id & 0 \\
		0 & 0 & \calM
	\end{bmatrix} \begin{bmatrix}
		\dot{w} \\
		\dot{\widetilde{u}}\\
		\dot{p}
	\end{bmatrix} &= \begin{bmatrix}
		0 & -\calA^{1/2} & \calD^* \\
		\calA^{1/2} & 0 & 0 \\
		-\calD & 0 & -\calK
	\end{bmatrix} \begin{bmatrix}
		w\\
		\widetilde{u}\\
		p
	\end{bmatrix}  + \begin{bmatrix}
		\Id & 0\\
		0 & 0\\
		0 & \Id
 	\end{bmatrix} \begin{bmatrix}
 		f\\g
 	\end{bmatrix},\\
 	\label{eqn:twoField:opMatrix4:output}
 	\widetilde{y} &= \begin{bmatrix}
 		\Id & 0 & 0\\
 		0 & 0 & \Id
 	\end{bmatrix}\begin{bmatrix}
 		w\\
 		\widetilde{u}\\
 		p
 	\end{bmatrix}.
\end{align}
\end{subequations}
The corresponding Hamiltonian reads
\begin{equation*}
%\label{eqn:energyParabolic}
  \widetilde{\mathcal{H}}(w,\widetilde{u},p)
  \vcentcolon= \frac{1}{2}\, \Big(\langle \calY w, w\rangle + \langle \widetilde{u}, \widetilde{u}\rangle + \langle \calM p, p\rangle \Big)
\end{equation*}
and is thus equal to the original energy defined in~\eqref{eqn:energyfull}. Further note that also the output coincides, i.e., $\widetilde{y} = y$.
In contrast to the formulation~\eqref{eqn:twoField:opMatrix3}, however, the pH-PDAE~\eqref{eqn:twoField:opMatrix4} only requires the usual regularity assumptions, namely~$\dot{u}\in L^2(\TT;\cHV)$.
\end{remark}
\begin{remark}
\label{rem:Dirichlet}
Also in the case of inhomogeneous Dirichlet boundary conditions, i.e., $\hat u_b\neq0$ or~$\hat p_b\neq0$ in~\eqref{eqn:poro:strong:c} and~\eqref{eqn:poro:strong:d}, we can obtain a pH formulation that explicitly encodes the boundary conditions. In more detail, following \cite{Alt15,Sim00}, we can use the trace-operator to write the boundary condition as an operator equation, which is then coupled to the original system via a Lagrange multiplier. To retain the pH structure, we have to choose a negative sign for the Lagrange multiplier and add the derivative of the boundary operator equation for the displacement $u$. From an DAE perspective, the latter corresponds to an {\em index reduction}. 
Denoting the respective trace operators by~$\calB_u$ and $\calB_p$, and $u_b$, $p_b$ defined analogously to $f$, $g$ as functionals acting on appropriate test functions, the pH formulation~\eqref{eqn:twoField:opMatrix3} extends to 
\begin{align*}
\begin{bmatrix}\calY & 0 & 0 & 0 & 0\ \ \\ 0 & \calA & 0 & 0 & 0\ \ \\ 0 & 0 & \calM & 0 & 0\ \ \\ 0 & 0 & 0 & 0 & 0\ \ \\ 0 & 0 & 0 & 0 & 0\ \ \end{bmatrix}
\begin{bmatrix} \dot{w} \\\dot{u}\\ \dot{p} \\ \dot{\lambda}_u \\ \dot{\lambda}_p \end{bmatrix}
= \begin{bmatrix} 0 & -\calA & \calD^* & \calB_u^* & 0 \\ \calA & 0 & 0 & 0 & 0 \\ -\calD & 0 & -\calK & 0 & \calB_p^* \\ -\calB_u & 0 & 0 & 0 & 0 \\ 0 & 0 & -\calB_p & 0 & 0 \end{bmatrix}
\begin{bmatrix} w\\ u\\ p\\ \lambda_u\\ \lambda_p \end{bmatrix}  + \begin{bmatrix} f\\ 0 \\ g \\ \dot u_b \\ p_b \end{bmatrix},
\end{align*}
With this choice, the additional terms for the boundary conditions enter the skew-symmetric operator matrix~$\mathcal{J}$. Moreover, the Dirichlet data acts as additional input thus allowing control via the boundary. 
\end{remark}
%
%
%%%%%%%%%%%%%%%%%%%%%%%%%%%%%%
\subsection{PH formulation of the quasi-static case}\label{sec:PH:quasiStatic}
First, we observe that the quasi-static case, i.e., $\rho \ddot u=0$, is included in the pH formulation~\eqref{eqn:twoField:opMatrix3}, by setting $\calY=0$.
It is worth noting that this strategy also works for the corresponding three-field formulation of poroelasticity, cf.~\cite{Biot41}. This formulation considers a mixed (or partitioned) version of the operator~$\calK$ with the fluid flux (or Darcy velocity) as an additional variable. 

Recall that the first-order formulation~\eqref{eqn:twoField:opMatrix3} results from introducing the derivative of~$u$ as an additional variable. If we directly start with the quasi-static case~\eqref{eqn:twoFieldQS} and do not introduce a new variable for the derivative of $u$, then we can obtain another pH-PDAE formulation.
We have already seen that one may eliminate the variable~$u$, which leads to the parabolic equation~\eqref{parabolic}. Also this has a pH structure but does not cover the full information due to the missing displacement variable. To reveal the structure of the entire system, we consider a reformulation. Since $\calK$ is invertible by assumption, we may introduce a new variable $q$ via
\begin{equation}
	\label{eqn:twoField:QS:additionalVariable}
 	\calK q =\calD u + \calM p.
\end{equation}
Assuming sufficient regularity of the given data, it is easy to see that the pairs $(u,p)$ and $(u,q)$ are equivalent in the sense that one can convert them into each other by
\begin{equation*}
  \begin{bmatrix} u\\ q \end{bmatrix}
  = \begin{bmatrix} \Id & 0\\ \calK^{-1}\calD & \calK^{-1}\calM \end{bmatrix}
  \begin{bmatrix} u \\ p \end{bmatrix}, \qquad
  \begin{bmatrix} u\\ p \end{bmatrix}
  = \begin{bmatrix} \Id & 0\\ -\calM^{-1}\calD & \calM^{-1}\calK \end{bmatrix}
  \begin{bmatrix} u \\ q \end{bmatrix}.
\end{equation*}
With this equivalence, we can write the quasi-static formulation~\eqref{eqn:twoFieldQS} equivalently in terms of the new variable, as
\begin{equation}
\label{eqn:twoField:matrix:ur}
  \begin{bmatrix} 0 & 0\\ 0 & \calK \end{bmatrix}
  \begin{bmatrix} \dot{u}\\ \dot{r} \end{bmatrix}
  = \begin{bmatrix} -\calA - \calD^*\calM^{-1}\calD & \calD^*\calM^{-1}\calK\\
     \calK\calM^{-1}\calD & -\calK\calM^{-1}\calK \end{bmatrix}
  \begin{bmatrix} u\\ r\end{bmatrix}
  + \begin{bmatrix} f\\ g \end{bmatrix}.
\end{equation}
The operator-matrix on the right-hand side is self-adjoint and negative semi-definite revealing that~\eqref{eqn:twoField:matrix:ur} is already in pH form. Note, however, that the operator~$\calK \calM^{-1}\calK$ occurs on the right-hand side, which requires strong spatial regularity and is thus not favorable. Instead, we observe that the system matrix in~\eqref{eqn:twoField:matrix:ur} is a Schur complement, i.e., we have the identity
\begin{equation*}
  \begin{bmatrix} -\calA & 0 & \calD^*\\ 0 & 0 & -\calK\\ -\calD & \calK & -\calM \end{bmatrix}
  \begin{bmatrix} \Id & 0 & 0\\ 0 & \Id & 0 \\
  -\calM^{-1}\calD & \calM^{-1}\calK & \Id \end{bmatrix}
  = \begin{bmatrix} -\calA - \calD^*\calM^{-1}\calD & \calD^*\calM^{-1}\calK & \calD^*\\
  \calK \calM^{-1}\calD & -\calK \calM^{-1}\calK & -\calK\\
  0 & 0 & -\calM \end{bmatrix}.
\end{equation*}
Hence, if we combine~\eqref{eqn:twoField:QS:additionalVariable} and~\eqref{eqn:twoField:matrix:ur} and rearrange the equations we obtain the extended system
\begin{equation}
\label{eqn:twoField:poroPH}
  \begin{bmatrix} 0 & 0 & 0\\  0 & 0 & 0\\  0 & 0 & \calK \end{bmatrix}
  \begin{bmatrix} \dot{u} \\ \dot{p}\\ \dot{q} \end{bmatrix}
  = \left(\begin{bmatrix} 0 &  \phantom{-}\calD^* & 0\\ -\calD & 0 & \calK\\ 0 & -\calK & 0 \end{bmatrix}
  - \begin{bmatrix} \calA & 0 & 0\\ 0 & \calM & 0\\ 0 & 0 & 0 \end{bmatrix}\right)
  \begin{bmatrix} u\\ p\\ q \end{bmatrix}
  + \begin{bmatrix} f\\ 0 \\ g\end{bmatrix}.
\end{equation}
We emphasize that this gives a pH-PDAE in the spirit of \Cref{def:pH-DAE} with Hamiltonian
\begin{equation}	
	\label{eqn:energyTwoField}
	\widehat{\mathcal{H}}(u,p,q)
	\vcentcolon= \frac{1}{2}\, \langle \calK q,q\rangle
	= \frac{1}{2}\, \langle\calM p + \calD u, q\rangle.
%	= \frac 12 \| \calD u + \calM p\|_{\calK^{-1}}^2.
\end{equation}
This can be verified immediately by noting that the operator matrices on the left and the second on the right are self-adjoint and positive semi-definite. Further, the first operator matrix on the right is skew-adjoint. Let us emphasize that the formulation~\eqref{eqn:twoField:poroPH} is not a pH-PDAE in the sense of our definition if $\calK$ is not self-adjoint.
\begin{remark}
	The pH-PDAE formulation~\eqref{eqn:twoField:poroPH} may also be obtained directly by adding the equation~\eqref{eqn:twoField:QS:additionalVariable} to the quasi-static formulation~\eqref{eqn:twoFieldQS} and using the time derivative of this equation to replace $\calD\dot{u} +\calM \dot{p}$ by $\calK \dot{q}$.
\end{remark}
%
%
%%%%%%%%%%%%%%%%%%%%%%%%%%%%%%
\subsection{PH formulation of the network case}\label{sec:PH:network}
In \Cref{sec:model:network} we have seen that the multi-field network case can be formulated in the same way as the two-field model if the operators are adjusted accordingly. More precisely, we simply need to replace the operators~$\calK$, $\calM$, and~$\calD$ by the operator matrices~$\calKK{B}$, $\calMM$, and~$\calDD$, respectively. By \Cref{ass:smallExchangeRates} we know that~$\calKK{B}$ is elliptic.
As a result, the pH formulation of~\eqref{eqn:network:opMatrix} is given by
\begin{equation}
\label{eqn:network:pH}
\begin{bmatrix}\calY & 0 & 0\\ 0 & \calA & 0 \\ 0 & 0 & \calMM\end{bmatrix}
\begin{bmatrix} \dot{w} \\\dot{u}\\ \dot{\mathbf{p}} \end{bmatrix}
= \begin{bmatrix} 0 & -\calA & \calDD^* \\ \calA & 0 & 0 \\ -\calDD & 0 & -\calKK{B} \end{bmatrix}
\begin{bmatrix} w\\ u\\ \mathbf{p} \end{bmatrix}  + \begin{bmatrix} f\\ 0 \\ \mathbf{g} \end{bmatrix}.
\end{equation}
As before, we may reinterpret the right-hand sides as external inputs, leading to the power-conjugated output consisting of~$w$ and~$\mathbf{p}$.
The corresponding Hamiltonian has the form
\begin{align*}
\mathcal{H}(w,u,\mathbf{p})
&\vcentcolon= \frac{1}{2}\, \Big(\langle \calY w, w\rangle + \langle \calA u, u\rangle + \langle\calMM \mathbf{p}, \mathbf{p}\rangle\Big)\\
&\phantom{\vcentcolon}= \frac{1}{2}\, \Big(\langle \calY w, w\rangle + \langle \calA u, u\rangle + \sum_{i=1}^m \langle\calM p_i, p_i\rangle\Big).
\end{align*}

Finally, the pH formulation of the quasi-static network case immediately follows by setting~$\calY=0$. %Alternatively, we can use the pH formulation for the quasi-static case derived in \Cref{sec:PH:quasiStatic}. Since the operator $\calKK{B}$ appears on the left-hand side, however, this formulation is only valid for a symmetric network coupling matrix $B$.
%
%
%%%%%%%%%%%%%%%%%%%%%%%%%%%%%%
\subsection{DAE structure and index}\label{sec:PH:DAE}
A spatial discretization of the operator equations~\eqref{eqn:twoField:opMatrix3}, \eqref{eqn:twoField:opMatrix4}, or \eqref{eqn:twoField:poroPH}, e.g., by the finite element method, yields a (finite-dimensional) pH-DAE. However, it is interesting to note that the DAE structure and, in particular, the \emph{differentiation index} vary in the different formulations. Recall, see e.g.~\cite{BreCP96,KunM06}, that a constant coefficient DAE $E\dot{z}= Az+ k$ with a regular pair $(E,A)$ (i.e.~$\det (\lambda E-A)$ is not identically zero) has differentiation index $\nu=0$ if $E$ is invertible. It has differentiation index one if $W^TAV$ is invertible, where $V$ is a matrix that spans the kernel of $E$ and $W$ is a matrix that spans the kernel of $E^T$. Otherwise, it has differentiation index~$\nu\ge 2$.

The spatial discretization of \eqref{eqn:twoField:opMatrix3} with discretization parameter $h$ yields the pH-DAE
\begin{equation}
	\label{eqn:twoField:opMatrix3:discrete}
	\begin{bmatrix}
		M_{\calY} & 0 & 0\\
		0 & K_{\mathcal{A}} & 0\\
		0 & 0 & M_{\mathcal{M}}
	\end{bmatrix}\begin{bmatrix}
		\dot{w}_h\\
		\dot{u}_h\\
		\dot{p}_h
	\end{bmatrix} = \begin{bmatrix}
		0 & -K_{\mathcal{A}} & D^T\\
		K_{\mathcal{A}} & 0 & 0\\
		-D & 0 & -K_{\mathcal{K}}
	\end{bmatrix}\begin{bmatrix}
		w_h\\
		u_h\\
		p_h
	\end{bmatrix} + \begin{bmatrix}
		f_h\\
		0\\
		g_h
	\end{bmatrix}.
\end{equation}
Here, $K_{\mathcal{A}}$ and $K_{\mathcal{K}}$ denote the stiffness matrices corresponding to the operators $\mathcal{A}$ and $\mathcal{K}$, respectively, and $M_{\calY}$, $M_{\mathcal{M}}$ are the mass matrices corresponding to $\calY$ and $\mathcal{M}$. Since the finite element method is a Galerkin projection, we immediately observe that these matrices are symmetric. Under reasonable assumptions on the spatial discretization, the matrices are even positive definite (in the case $\rho> 0$).

If $\rho > 0$, then~\eqref{eqn:twoField:opMatrix3:discrete} is an implicit equation with nonsingular matrix on the left-hand side. Thus, it is a pH-DAE of differentiation-index $\nu=0$. If~$\rho \ddot u=0$, then the operator $\calY$ and the corresponding mass matrix $M_{\calY}$ are zero, i.e., the space-discretized system~\eqref{eqn:twoField:opMatrix3:discrete} reads
\begin{subequations}
	\label{eqn:twoField:opMatrix3:discrete:descriptor}
\begin{align}
	\label{eqn:twoField:opMatrix3:discrete:DAE}
	\begin{bmatrix}
		0 & 0 & 0\\
		0 & K_{\mathcal{A}} & 0\\
		0 & 0 & M_{\mathcal{M}}
	\end{bmatrix}\begin{bmatrix}
		\dot{w}_h\\
		\dot{u}_h\\
		\dot{p}_h
	\end{bmatrix} &= \begin{bmatrix}
		0 & -K_{\mathcal{A}} & D^T\\
		K_{\mathcal{A}} & 0 & 0\\
		-D & 0 & -K_{\mathcal{K}}
	\end{bmatrix}\begin{bmatrix}
		w_h\\
		u_h\\
		p_h
	\end{bmatrix} + \begin{bmatrix}
		M_u & 0\\
		0 & 0\\
		0 & M_p
	\end{bmatrix}\begin{bmatrix}
		\inpVar_{u,h}\\
		\inpVar_{p,h}
	\end{bmatrix},\\
	\begin{bmatrix}
		\outVar_{u,h}\\
		\outVar_{p,h}
	\end{bmatrix} &= \begin{bmatrix}
		M_u & 0 & 0\\
		0 & 0 & M_p
	\end{bmatrix}\begin{bmatrix}
		w_h\\
		u_h\\
		p_h
	\end{bmatrix},
\end{align}
\end{subequations}
where we have interpreted the external forcing as control variables and, as before, added the power-conjugated output equation. As before, the mass matrices $M_u$ and $M_p$ are assumed to be symmetric positive definite. Since the upper-left block in the matrix on the right-hand side of \eqref{eqn:twoField:opMatrix3:discrete:DAE} is zero, the pH-DAE cannot be of index $\nu=1$. Multiplication of the second and third equation with $K_{\mathcal{A}}^{-1}$ and $M_{\mathcal{M}}^{-1}$, rearranging the equations and variables details that~\eqref{eqn:twoField:opMatrix3:discrete:DAE} is in Hessenberg form, see \cite{BreCP96}. We immediately conclude that~\eqref{eqn:twoField:opMatrix3:discrete:DAE} has differentiation index $\nu=2$. The constraint equation $K_{\mathcal{A}}u_h = D^Tp_h + f_h$ is thus complemented with the \emph{hidden constraint}, i.e., the constraint that arises from a linear combination of the original equations and their time derivatives. In the present case, we obtain with~\eqref{eqn:twoField:opMatrix3:discrete:DAE} the hidden constraint 
\begin{equation*}
	(K_{\mathcal{A}} + D^TM_{\mathcal{M}}^{-1}D)w_h + D^TM_{\mathcal{M}}^{-1}K_{\mathcal{K}}p_h + \dot{f}_h = D^TM_{\mathcal{M}}^{-1}g_h.
\end{equation*}
As a direct consequence, initial values $w_{h}^0$ and $u_{h}^0$ for $w_h$ and $u_h$, respectively, have to satisfy the consistency conditions
\begin{align*}
	K_{\mathcal{A}}u_{h}^0 &= D^Tp_{h}^0 + f_h(0),\\
	(K_{\mathcal{A}} + D^TM_{\mathcal{M}}^{-1}D)w_{h}^0 &= -D^TM_{\mathcal{M}}^{-1}K_{\mathcal{K}}p_{h}^0 - \dot{f}_h(0) + D^TM_{\mathcal{M}}^{-1}g_h(0),
\end{align*}
showing that it is sufficient to prescribe an initial value $p_{h}^0$ for $p_h$. We emphasize that this is not specific to the discretized pH-DAE, but also applies to the pH-PDAE~\eqref{eqn:twoField:opMatrix3} with $\mathcal Y = 0$. In this case, the initial values $w^0$ and $u^0$ are given implicitly by
\begin{align}
	\label{eqn:consistencyCondition:u}
	\calA u^0 &= \calD^*p^0 + f(0),\\
	\label{eqn:consistencyCondition:w}
	(\calA + \calD^*\calM^{-1}\calD) w^0 &= -\calD^*\calM^{-1}\calK p^0 - \dot{f}(0) + \calD^*\calM^{-1}g(0).
\end{align}
\begin{remark}
Although the semi-discrete pH-DAE~\eqref{eqn:twoField:opMatrix3:discrete:descriptor} has differentiation index $\nu=2$, the index can be reduced to $\nu=1$ by output feedback. This is a standard procedure in DAE control to avoid impulsive solutions when the chosen input functions are not sufficiently smooth, see e.g.~\cite{BunBMN99,Dai89}. 
In more detail, consider a feedback
	\begin{displaymath}
		\begin{bmatrix}
			\inpVar_{u,h}\\
			\inpVar_{p,h}
		\end{bmatrix} = \begin{bmatrix}
			F_{11} & 0\\
			0 & 0
		\end{bmatrix}\begin{bmatrix}
			\outVar_{u,h}\\
			\outVar_{p,h}
		\end{bmatrix} + \begin{bmatrix}
			\tilde{\inpVar}_{u,h}\\
			\tilde{\inpVar}_{p,h}
		\end{bmatrix}.
	\end{displaymath}
	The resulting closed-loop system is given by
	\begin{equation}
		\label{eqn:twoField:opMatrix3:discrete:DAE:closedLoop}
		\begin{bmatrix}
		0 & 0 & 0\\
		0 & K_{\mathcal{A}} & 0\\
		0 & 0 & M_{\mathcal{M}}
	\end{bmatrix}\begin{bmatrix}
		\dot{w}_h\\
		\dot{u}_h\\
		\dot{p}_h
	\end{bmatrix} = \begin{bmatrix}
		M_u^T F_{11} M_u & -K_{\mathcal{A}} & D^T\\
		K_{\mathcal{A}} & 0 & 0\\
		-D & 0 & -K_{\mathcal{K}}
	\end{bmatrix}\begin{bmatrix}
		w_h\\
		u_h\\
		p_h
	\end{bmatrix} + \begin{bmatrix}
		\mathrm{Id} & 0\\
		0 & 0\\
		0 & \mathrm{Id}
	\end{bmatrix}\begin{bmatrix}
		\tilde{\inpVar}_{u,h}\\
		\tilde{\inpVar}_{p,h}
	\end{bmatrix},
	\end{equation}
	showing that~\eqref{eqn:twoField:opMatrix3:discrete:DAE:closedLoop} has differentiation index $\nu=1$, whenever $F_{11}$ is nonsingular, and has pH structure, whenever the symmetric part of~$F_{11}$ is negative semi-definite.
\end{remark}
If one does not introduce the  new variable $w$, then it has been shown in \cite{AltMU20} that the spatial discretization of the original quasi-static formulation~\eqref{eqn:twoFieldQS} results in a system of differentiation index~$\nu=1$. We emphasize that the original quasi-static formulation~\eqref{eqn:twoFieldQS} only encodes the consistency condition~\eqref{eqn:consistencyCondition:u} for $u$ but not the consistency condition~\eqref{eqn:consistencyCondition:w} for~$w$ (respectively $\dot{u}$).

\begin{remark}
\label{rem:high order formulation}
If the inhomogeneity $f$ is sufficiently smooth in time, then there is also another pH-PDAE like formulation. Going back to the original second-order model~\eqref{eqn:twoField:opMatrix} and differentiating the first equation with respect to time yields the third-order (in time) equation
\begin{subequations}
\label{eqn:twoField:operator3rd}
\begin{alignat}{6}
\calY u^{(3)}&\ +\ &\calA \dot{u} &\ -\ &\calD^* \dot{p} & & &\ =\ &\dot{f} &&\qquad \text{in } \V^*, \label{eqn:twoField:operator3rd:a}\\
 & & \calD \dot{u} &\ +\ & \calM\dot{p} &\ +\ &\calK p &\ =\ &g &&\qquad \text{in } \Q^*. \label{eqn:twoField:operator3rd:b}
\end{alignat}
\end{subequations}
This is a formulation where all coefficients are positive semi-definite, except for one which has a positive semi-definite symmetric part. This is the structure of the higher-order dissipative Hamiltonian systems  discussed in \cite{MehMW20}. In principle, we could work directly with this formulation and do not transform to first order, but we will not discuss this formulation further here.
\end{remark}
%
%
%=============================================================================
%=========  Interconnections
%=============================================================================
\section{Interconnection of subsystems}\label{sec:interconn}
It has been shown in \cite{MehM19}, see also \Cref{sec:PH}, that the (energy-preserving or dissipative) interconnection of pH-DAEs is again a pH-DAE. In the following, we show how the presented pH-PDAE formulations can be obtained via the interconnection of subsystems  via a suitable output feedback relation.
%
%
%%%%%%%%%%%%%%%%%%%%%%%%%%%%%%
\subsection{Interconnection in the two-field case}\label{sec:intercon:2field}
In the two-field case, we construct an interconnection of a hyperbolic (or elliptic) and a parabolic equation. We treat the two cases~$\rho \ddot u\neq 0$ and $\rho \ddot u=0$ together.
First, consider the system
\begin{equation}
	\label{eqn:hyperbolicPH}
	\calY \ddot u = -\calA u + f,
\end{equation}
which, by going to the first-order formulation and adding an output equation, can be written as
\begin{align}
\label{eqn:twoFieldhyperbolic}
\begin{bmatrix}\calY & 0 \\ 0 & \calA  \end{bmatrix}
  \begin{bmatrix} \dot{w} \\\dot{u} \end{bmatrix}
  = \begin{bmatrix} 0 & -\calA  \\ \calA & 0 \end{bmatrix}
  \begin{bmatrix} w\\ u\end{bmatrix}  + \begin{bmatrix} \inpVar_{u} \\ 0 \end{bmatrix},\qquad
  \outVar_u
  = \begin{bmatrix}  \Id & 0 \end{bmatrix}
  \begin{bmatrix}  w\\ u \end{bmatrix}
  = w.
\end{align}
For the system in $p$ we consider the (parabolic) system
\begin{equation}
\label{eqn:twoFieldparabolic}
\calM \dot{p} = -\calK p +\inpVar_p, \qquad \outVar_p=p.
\end{equation}
Both systems have a pH structure in the spirit of \Cref{def:pH-DAE} with Hamiltonians
\begin{equation}
\label{eqn:energyhyp}
  \mathcal{H}_\text{u}(w,u)
  \vcentcolon= \frac{1}{2}\, \Big(\langle \calY w, w\rangle+ \langle \calA u, u\rangle\Big)\qquad\text{and}\qquad
  \mathcal{H}_\text{p}(p)
  \vcentcolon= \frac{1}{2}\, \langle \calM p, p\rangle.
\end{equation}
\begin{theorem}
	\label{thm:phcoupling}
	The pH-PDAE~\eqref{eqn:twoField:opMatrix3} is obtained from \eqref{eqn:twoFieldhyperbolic} and \eqref{eqn:twoFieldparabolic} via the output feedback
	\begin{equation*}
		%\label{eqn:twoField:interconnection}
		\begin{bmatrix}
			\inpVar_u \\ \inpVar_p
		\end{bmatrix} = \begin{bmatrix}
			0  &  \calD^*\\
			-\calD  & 0
		\end{bmatrix}\begin{bmatrix}
			\outVar_u\\ \outVar_p
		\end{bmatrix}
+  \begin{bmatrix} f \\ g\end{bmatrix}.
	\end{equation*}
The Hamiltonian~\eqref{eqn:energyfull}  of the coupled system is the sum of the Hamiltonians in~\eqref{eqn:energyhyp}.
\end{theorem}
%
%%%%%%%%%%%%%%%%%%%%%%%%%%%%%%
\subsection{Interconnection for the alternative quasi-static case}\label{sec:intercon:2fieldQS}
If we are solely interested in the quasi-static case, i.e., $\rho \ddot u=0$ and $\calY = 0$, then~\eqref{eqn:hyperbolicPH} reduces to the elliptic equation
\begin{equation}
	\label{eqn:ellipticPH}
	 -\calA u + \inpVar_{u}=0.
\end{equation}
In this case, we may choose the output as~$\outVar_u=u$. In terms of \Cref{def:pH-DAE}, respectively \Cref{lem:pHequivalent}, we have~$\calR = \calA$, $\calB = \Id$, and zero operators otherwise. The associated Hamiltonian in this case is constant in time and we may for simplicity choose $\mathcal{H}_\text{u}(u) = 0$.

To obtain the alternative pH formulation presented in \Cref{sec:PH:quasiStatic}, we consider for $p$ the last two equations in \eqref{eqn:twoField:poroPH}, given by
\begin{subequations}
	\label{eqn:parabolicPH}
	\begin{align}
 	\begin{bmatrix}
 		0 & 0\\
 		0 & \calK
 	\end{bmatrix}\begin{bmatrix}
 		\dot{p}\\\dot{q}
 	\end{bmatrix} &= \left(\begin{bmatrix}
 		0 & \calK\\
 		-\calK & 0
 	\end{bmatrix} - \begin{bmatrix}
 		\calM & 0\\
 		0 & 0
 	\end{bmatrix}\right)\begin{bmatrix}
 		p\\ q
 	\end{bmatrix} + \begin{bmatrix}
 		\inpVar_{p}\\
 		\inpVar_{q}
 	\end{bmatrix},\\
 	\begin{bmatrix}
 		\outVar_{p}\\
 		\outVar_{q}
 	\end{bmatrix} &= \begin{bmatrix}
 		p\\ q
 	\end{bmatrix}
	\end{align}
\end{subequations}
with associated Hamiltonian $\mathcal H_p(p,q) = \tfrac{1}{2} \langle \calK q,q\rangle$. Choosing a suitable output feedback for $\inpVar_{u}$, $\inpVar_p$, $\inpVar_p$, we recover \eqref{eqn:twoField:poroPH} as stated in the following immediate result.
\begin{theorem}
	\label{thm:phcoupling:QS}
	The alternative pH formulation~\eqref{eqn:twoField:poroPH} for the quasi-static model~\eqref{eqn:twoFieldQS} is obtained by the output feedback
	\begin{equation*}
		%\label{eqn:twoField:interconnection:QS}
		\begin{bmatrix}
			\inpVar_{u} \\  \inpVar_{p}\\ \inpVar_{q}
		\end{bmatrix} = \begin{bmatrix}
			 0 &  \calD^*& 0\\
                -\calD  & 0 & 0\\
			0 &0 & 0
		\end{bmatrix}\begin{bmatrix}
			\outVar_{u}\\\outVar_{p}\\\outVar_{q}
		\end{bmatrix}
+  \begin{bmatrix}  f \\ 0 \\ g\end{bmatrix}.
	\end{equation*}
	The Hamiltonian~\eqref{eqn:energyTwoField} of the coupled system is the sum of the Hamiltonians of the subsystems.
\end{theorem}
%
%
%%%%%%%%%%%%%%%%%%%%%%%%%%%%%%
\subsection{Interconnection in the network case}
\label{sec:intercon:network}
Clearly, we obtain all the different pH-PDAE formulations similarly as in the two-field case, since the multi-field network structure is precisely the same as that of the two-field system, provided that~\Cref{ass:smallExchangeRates} holds, i.e., we have small exchange rates. To perform the interconnection via output feedback, replace the parabolic equation~\eqref{eqn:twoFieldparabolic} with
\begin{equation}
\label{eqn:network:parabolicrho}
\calM\dot{p}_i=-\calK_i p_i + \inpVar_{p_i},\qquad \outVar_{p_i}=  p_i
\end{equation}
for $i=1,\ldots,m$ with associated Hamiltonians $\mathcal{H}_i(p_i) \vcentcolon= \tfrac{1}{2} \langle \calM p_i, p_i\rangle$. Collecting the input and output variables in $\mathbf{v}_{\mathbf{p}} = [v_{p_1};\ldots;v_{p_m}]$ and $\mathbf{v}_{\mathbf{p}} = [v_{p_1};\ldots;v_{p_m}]$, respectively, the feedback interconnection is -- similar to \Cref{thm:phcoupling} -- given by
	\begin{align*}
		\begin{bmatrix}
			\inpVar_{u}\\
			\mathbf{\inpVar}_{\mathbf{p}}
		\end{bmatrix} = \left(\begin{bmatrix}
			0 & \calDD^*\\
			-\calDD & B_{\mathrm{skew}}\otimes\calI
		\end{bmatrix} + \begin{bmatrix}
			0 & 0\\
			0 & B_{\mathrm{sym}}\otimes\calI
		\end{bmatrix}\right)\begin{bmatrix}
			\outVar_u\\
			\mathbf{\outVar}_{\mathbf{p}}
		\end{bmatrix} + \begin{bmatrix}
			f\\
			\mathbf{g}
		\end{bmatrix}.
	\end{align*}
Note that~\Cref{ass:smallExchangeRates} ensures that the dissipation operator $\mathcal{R}$ remains positive semi-definite. In fact, \Cref{ass:smallExchangeRates} can be weakened in the sense that it is sufficient that the symmetric part~$B_{\mathrm{sym}}$ of the exchange rate matrix $B$ is small enough in the sense of positive semi-definite operators, cf.~\Cref{lem:smallExchangeRatesSymBeta}.

If the exchange rates are symmetric, i.e., $B = B_{\mathrm{sym}}$, then we can also use the alternative quasi-static pH formulation derived in \Cref{sec:PH:quasiStatic}. Since the exchange rates, encoded in the matrix $B$, then appear on the left-hand side in the pH formulation~\eqref{eqn:twoField:poroPH}, the pH formulation cannot be obtained directly via feedback interconnection. Instead, we can first couple the pressure equations~\eqref{eqn:network:parabolicrho}, then introduce the new variable $\mathbf{q}$ as in \Cref{sec:PH:quasiStatic}, and afterwards obtain the final form with the coupling from \Cref{thm:phcoupling:QS}. An illustration of the two pH formulations for the quasi-static network case is presented in \Cref{fig:subsystems}.
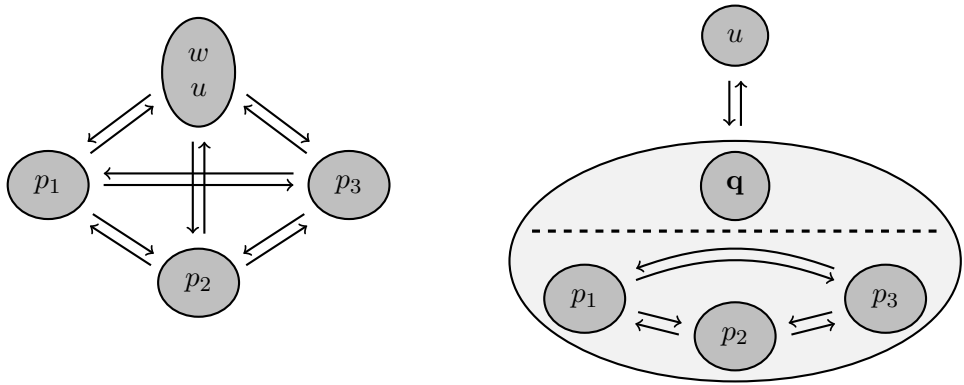
\begin{figure}
	\centering
	\begin{subfigure}[t]{.45\linewidth}
		\centering
		\begin{tikzpicture}
		\tikzstyle{myEllipse} = [ellipse,outer sep=.2cm,inner sep=0.2cm, draw, thick, fill=black!25]

		%\draw[fill=black!5,thick] (2.7,0) ellipse (5 and 2.5);
		\node (u) [myEllipse,align=center] at (0,0) {$w$\\$u$};
		\node (sys1) [myEllipse,align=center] at (-2,-1.5) {$p_1$};
		\node (sys2) [myEllipse,align=center] at (0,-2.8) {$p_2$};
		\node (sys3) [myEllipse,align=center] at (2,-1.5) {$p_3$};
		
		\draw[->, thick] (sys1) -- (sys2);
		\draw[->, thick, transform canvas={xshift=-.1em, yshift=-.4em}] (sys2) -- (sys1);
		
		\draw[->, thick] (sys3) -- (sys2);
		\draw[->, thick, transform canvas={xshift=.1em, yshift=-.4em}] (sys2) -- (sys3);
		
		\draw [->,thick] (sys1) to (sys3);
		\draw [->,thick,transform canvas={yshift=.4em}] (sys3) to (sys1);
		
		\draw[->, thick] (sys1) -- (u);
		\draw[->, thick, transform canvas={xshift=-.1em, yshift=.4em}] (u) -- (sys1);
		
		\draw[->, thick, transform canvas={xshift=0.2em}] (sys2) -- (u);
		\draw[->, thick, transform canvas={xshift=-0.2em}] (u) -- (sys2);
		
		\draw[->, thick] (sys3) -- (u);
		\draw[->, thick, transform canvas={xshift=.1em, yshift=.4em}] (u) -- (sys3);
		
		\node at (0,-4) {}; % shift of the image
		\end{tikzpicture}
		\caption{Coupling pattern based on \Cref{sec:PH:twoField} in the case~$\rho\ddot u=0$.}
		\label{fig:subsystems:feedback}
	\end{subfigure}\quad
	\begin{subfigure}[t]{.45\linewidth}
		\centering
		\begin{tikzpicture}
		\tikzstyle{myEllipse} = [ellipse,outer sep=.2cm,inner sep=0.2cm, draw, thick, fill=black!25]

		%\draw[fill=black!5,thick] (2.7,0) ellipse (5 and 2.5);
		
		%\draw[fill=black!5,thick] (-3,-4.6) rectangle (3,-1.4);
		%\draw[very thick,dashed] (-2.8,-2.6) -- (2.8,-2.6);
		
		\node (rpSys) [myEllipse,fill=black!5,minimum width=6cm,minimum height=3.2cm] at (0,-3) {};
		\draw[very thick,dashed] (-2.7,-2.6) -- (2.7,-2.6);
		
		\node (u) [myEllipse] at (0,0) {$u$};
		\node (sys1) [myEllipse] at (-2,-3.5)  {$p_1$};
		\node (sys2) [myEllipse] at (0,-4) {$p_2$};
		\node (sys3) [myEllipse] at (2,-3.5) {$p_3$};
		\node (r) [myEllipse] at (0,-2) {$\mathbf{q}$};
		
		\draw[->, thick] (sys1) -- (sys2);
		\draw[->, thick, transform canvas={xshift=-.1em, yshift=-.4em}] (sys2) -- (sys1);
		
		\draw[->, thick] (sys3) -- (sys2);
		\draw[->, thick, transform canvas={xshift=.1em, yshift=-.4em}] (sys2) -- (sys3);
		
		\draw [->,thick] (sys1) to [out=20,in=160] (sys3);
		\draw [->,thick,transform canvas={yshift=.4em}] (sys3) to [out=160,in=20] (sys1);
		
		\draw[->, thick, transform canvas={xshift=0.2em}] (rpSys) -- (u);
		\draw[->, thick, transform canvas={xshift=-0.2em}] (u) -- (rpSys);
		\end{tikzpicture}
		\caption{Coupling pattern based on the alternative formulation of \Cref{sec:PH:quasiStatic}.}
		\label{fig:subsystems:directFormulation}
	\end{subfigure}
	\caption{Illustration of different pH formulations via feedback interconnection for the quasi-static poroelastic network.}
	\label{fig:subsystems}
\end{figure}
%
%
%=============================================================================
%=========  Conclusions
%=============================================================================
\section{Summary}\label{sec:conclusion}
We have studied Biot's poroelasticity model from an energy-based perspective and introduced different pH formulations. % of the PDAE model.
Our operator matrices mimic the finite-dimensional definition of pH-DAEs such that a spatial discretization via finite elements yields a pH-DAE. To obtain the required properties, we have used the commonly omitted second-order term and performed a first-order reformulation. The quasi-static pH formulation is then obtained by formally setting the second-order term zero, which does not affect the pH structure.
The formulation as pH-PDAE is favorable because it automatically induces a dissipation inequality and is invariant under Galerkin projection, making our formulation amendable to structure-preserving discretization and model-order reduction. Besides, the pH formulation offers a system-theoretic interpretation of poroelastic network models. In particular, we have shown that the pH formulation of the network models can be obtained via output feedback interconnection of the different submodels with internal energies. We emphasize that such an interconnection is not directly possible with the quasi-static model typically studied in the literature, since the coupling of the subsystems requires the derivative of the displacement.
%
%
%=============================================================================
%=========  Acknowledgements
%=============================================================================
\section*{Acknowledgements}
The work of V. Mehrmann is supported by the Deutsche Forschungsgemeinschaft (DFG, German Research Foundation) CRC 910 \emph{Control of self-organizing nonlinear systems: Theoretical methods and concepts of application}, 163436311. B. Unger acknowledges funding from the DFG under Germany's Excellence Strategy -- EXC 2075 -- 390740016 and is thankful for support by the Stuttgart Center for Simulation Science (SimTech).

We thank Etienne Emmrich for very helpful discussions on the analytical properties of the system.
%
%
%=============================================================================
%=========  Bibs
%=============================================================================
\bibliographystyle{plain}
\bibliography{references}
\end{document}